\documentclass[oneside]{amsart}

\usepackage{amssymb}
\usepackage{amsthm}
\usepackage{amscd}
\usepackage{enumitem}

\numberwithin{equation}{section}
\theoremstyle{plain}

\newtheorem{thm}{Theorem}[section]

\newtheorem{lemma}[thm]{Lemma}

\theoremstyle{definition}

\newcommand{\dlabel}[1]{\ifmmode \text{\ttfamily \upshape [#1] } \else
{\ttfamily \upshape [#1] }\fi \label{#1}}

\begin{document}

\setlength{\baselineskip}{15pt}

\title{Finite $p$-groups having Schur multiplier of maximum order}

\author{Sumana Hatui}
\address{School of Mathematics, Harish-Chandra Research Institute, Chhatnag Road, Jhunsi, Allahabad 211019, INDIA}
\address{\& Homi Bhabha National Institute, Training School Complex, Anushakti Nagar, Mumbai 400085, India}
\email{sumanahatui@hri.res.in, sumana.iitg@gmail.com}


\begin{abstract}
Let $G$ be a non-abelian $p$-group of order $p^n$ and $M(G)$ denote the Schur multiplier of $G$. Niroomand proved that $|M(G)| \leq p^{\frac{1}{2}(n+k-2)(n-k-1)+1}$ for non-abelian $p$-groups $G$ of order $p^n$ with derived subgroup of order $p^k$. Recently Rai classified $p$-groups $G$ of nilpotency class $2$ for which $|M(G)|$ attains this bound. In this article we show that there is no finite $p$-group $G$ of nilpotency class $c \geq 3$ for $p\neq3$ such that $|M(G)|$ attains this bound. Hence $|M(G)| \leq p^{\frac{1}{2}(n+k-2)(n-k-1)}$ for $p$-groups $G$ of class $c \geq 3$ where $p \neq 3$.
We also construct a $p$-group $G$ for $p=3$ such that $|M(G)|$ attains the Niroomand's bound.
\end{abstract}

\maketitle

\section{Introduction}
Let $G$ be a finite $p$-group and $G'$ denotes its derived subgroup. The concept of Schur multiplier $M(G)$ of a group $G$ was introduced by Schur \cite{IS1} in 1904 while studying of projective representation of groups. In 1956, Green \cite{JG} proved that $|M(G)| \leq p^{\frac{1}{2}n(n-1)}$ for $p$-groups $G$ of order $p^n$. So $|M(G)|=p^{\frac{1}{2}n(n-1)-t(G)}$, for some $t(G) \geq 0$.
It is of interest to characterize the structure of all non-abelian $p$-groups $G$ by the order of the Schur multiplier $M(G)$, i.e., when $t(G)$ is known. This problem was studied by several authors for various values of $t(G)$ and the results are known for $0 \leq t(G) \leq 6$ in \cite{BY,ZH,EG,PN3,PN1, SIX}.

Later Niroomand \cite{PN} improved the Green's bound and showed that for non abelian $p$-groups $G$ of order $p^n$ with $|G'|=p^k$, 

\centerline{$|M(G)|\leq p^{\frac{1}{2}(n+k-2)(n-k-1)+1}$.} 
He also classified groups $G$ such that $|M(G)|=p^{\frac{1}{2}(n+k-2)(n-k-1)+1}$ with $k=1$. For our convention instead of writing $|M(G)|=p^{\frac{1}{2}(n+k-2)(n-k-1)+1}$ we shall write $|M(G)|$ attains the bound, throughout this paper.

Recently Rai \cite[Theorem 2.1]{PK} classified finite $p$-groups $G$ of class $2$ such that $|M(G)|$ attains the bound.

Aim of this paper is to continue this line of investigation and to look into the classification of arbitrary finite $p$-groups attending this bound. It, surprisingly turns out that for $p \neq 3$ there is no finite $p$-group $G$ of nilpotency class $c \geq 3$ such that $|M(G)|$ attains the bound. Hence for $p$-groups $G$ of class $\geq 3$ and $p \neq 3$ we improve the bound and in this case $|M(G)| \leq p^{\frac{1}{2}(n+k-2)(n-k-1)}$ where $|G'|=p^k$.
Now one can ask what will happen for $p=3$? Is this above statement true for $p=3$?\\
The answer to this question is no. We construct an example, see Lemma \ref{C2}, which gives the answer of this question. Now we state our main result which is the following.
\begin{thm}\label{C1}
There is no non-abelian $p$-group $G$ of order $p^n$, $p\neq3$, having nilpotency class $c \geq 3$ with $|G'|=p^k$ and $|M(G)|=p^{\frac{1}{2}(n+k-2)(n-k-1)+1}$. In particular, $|M(G)| \leq p^{\frac{1}{2}(n+k-2)(n-k-1)}$ for $p$-groups $G$ of nilpotency class $c \geq 3$ and $p \neq 3$.
\end{thm}

\noindent So the natural question arises here which is the following:\\
{\bf Question:} Does there exist finite $p$-groups of arbitrary nilpotency class for which this new bound is attained?\\
The answer of this question is yes for nilpotency class $3$ and $4$, see Section 3, Example1 and Example2.
We conclude this section by defining some notations. We say a group $G$ of order $p^n$ is of maximal class if the nilpotency class of $G$ is $n-1$.
By $ES_p(p^3)$ we denote the extra-special $p$-group of order $p^3$ having exponent $p$. By $\mathbb{Z}_p^{(k)}$ we denote $\mathbb{Z}_p \times \mathbb{Z}_p \times \cdots \times \mathbb{Z}_p$($k$ times). $\gamma_i(G)$ denotes the $i$-th term of the lower central series of group $G$ and $G^{ab}$ denotes the quotient group $G/\gamma_2(G)$.
\section{Preliminaries}
In this section we recall some results which will be used to prove our main result.
\begin{thm}$($\cite[Theorem 4.1]{MRRR}$)$\label{J}
Let $G$ be a finite group and $K$ a central subgroup of $G$. Set $A = G/K$. Then
$|M(G)||G'\cap K|$ divides $|M(A)| |M(K)| |A^{ab} \otimes K|$.
\end{thm}
\begin{thm}$($\cite[Main Theorem]{PN}$)$\label{N}
Let $G$ be a non-abelian finite $p$-group of order $p^n$. If $|G'|=p^k$, then we have \\
\centerline{$|M(G)| \leq p^{\frac{1}{2}(n+k-2)(n-k-1)+1}$.}
In particular, \\
\centerline{$|M(G)| \leq p^{\frac{1}{2}(n-1)(n-2)+1}$,}
and the equality holds in this last bound if and only if $G = H \times Z$, where $H$ is an extra special $p$-group of order $p^3$ and exponent $p$, and $Z$ is an elementary abelian $p$-group.
\end{thm}
The following result gives the classification of $p$-groups $G$ of class $2$ such that $|M(G)|$ attains the bound.
\begin{thm}$($\cite[Theorem 1.1]{PK}$)$\label{Q}
Let $G$ be a finite p-group of order $p^n$ and nilpotency class $2$ with $|G'| = p^k$. Then $|M(G)|=p^{\frac{1}{2}(n+k-2)(n-k-1)+1}$ if and only if $G$ is one of
the following groups.

$(i)$ $G_1=ES_p(p^3) \times \mathbb{Z}_p^{(n-3)}$, for an odd prime $p$.

$(ii)$ $G_2= \langle{\alpha,\alpha_1,\alpha_2,\beta_1,\beta_2 \mid [\alpha_i,\alpha]=\beta_i, \alpha^p=\alpha_i^p=\beta_i^p=1 (i=1,2)\rangle}$, for an odd prime $p$.

$(iii)$ $G_3=\langle \alpha_1,\beta_1,\alpha_2,\beta_2,\alpha_3, \beta_3 \mid [\alpha_1,\alpha_2]=\beta_3, [\alpha_2,\alpha_3]=\beta_1,$ $[\alpha_3,\alpha_1]=\beta_2,\alpha_i^{(p)}=\beta_i^p=1(i=1,2,3)\rangle$, for an odd prime $p$.
\end{thm}
The following result is a consequence of \cite[Main Theorem]{SX}.
\begin{thm}\label{S}
There is no group $G$ of order $p^n$ and of class $c\geq 3$ such that $|M(G)|=p^{\frac{1}{2}n(n-1)-(n+1)}$ where $n \geq 6$, $p$ is odd.
\end{thm}
The following result immediately follows from \cite[Theorem 21]{PN2}.
\begin{thm}\label{P}
There is no group $G$ of order $p^n$ and of class $c\geq 3$ such that $|M(G)|=p^{\frac{1}{2}n(n-1)-(n-1)}$.
\end{thm}
Let $G$ be a finite $p$-group of nilpotency class $3$ with centre $Z(G)$. Set $\bar{G}=G/Z(G)$. 
Define a homomorphism $\psi_2 :\bar{G}^{ab} \otimes \bar{G}^{ab} \otimes \bar{G}^{ab}  \rightarrow \frac{\gamma_2(G)}{\gamma_3(G)} \otimes \bar{G}^{ab}$ such that \\
$\psi_2(\bar{x}_1 \otimes \bar{x}_2 \otimes \bar{x}_3)=[x_1,x_2]_{\gamma} \otimes \bar{x}_3 + [x_2,x_3]_{\gamma} \otimes \bar{x}_1 + [x_3,x_1]_{\gamma} \otimes \bar{x}_2$, where  $\bar{x}$ denotes the image in $\bar{G}$ of the element $x \in G$ and $[x,y]_{\gamma}$ denotes the image in $\frac{\gamma_2(G)}{\gamma_3(G)}$ of the commutator $[x,y] \in G$. Define another homomorphism $\psi_3 : \bar{G}^{ab} \otimes \bar{G}^{ab} \otimes \bar{G}^{ab} \otimes \bar{G}^{ab} \rightarrow \gamma_3(G) \otimes \bar{G}^{ab}$ such that\\
$\psi_3(\bar{x}_1 \otimes \bar{x}_2 \otimes \bar{x}_3 \otimes \bar{x}_4)=[[x_1,x_2],x_3] \otimes \bar{x}_4 + [x_4,[x_1,x_2]] \otimes \bar{x}_3 + [[x_3,x_4],x_1]\otimes \bar{x}_2$ $ +[x_2,[x_3,x_4]] \otimes \bar{x}_1.$
\begin{thm}$($\cite[Proposition 1]{EW} and \cite{E}$)$\label{RM}
Let $G$ be a finite $p$-group of nilpotency class $3$. With the notations above, we have \\
\centerline{$|M(G)||\gamma_2 (G)||Image(\psi_2)||Image(\psi_3)| \leq |M(G^{ab})||\frac{\gamma_2(G)}{\gamma_3(G)} \otimes \bar{G}^{ab}||\gamma_3(G) \otimes \bar{G}^{ab}|$.}
\end{thm}
\section{Proof of Main Theorem}
In this section we prove our main theorem. Before going to the proof of main theorem, we need some reults. We use Theorem \ref{N} and Theorem \ref{J} in the proof of the following lemma without further reference.
\begin{lemma}\label{R}
Let $G$ be a non-abelian $p$-group of order $p^n$ with $|G'|=p^k$ and $|M(G)|$ attains the bound. Then the following hold:

$(i)$ $G^{ab}$ is an elementary abelian $p$-group.

$(ii)$ $Z(G)$ is an elementary abelian $p$-group.

$(iii)$ $Z(G) \subseteq G'$ $($except $G \cong ES_p(p^3) \times \mathbb{Z}_p^{(n-3)})$.
\end{lemma}
\begin{proof}
$(i)$ Let $K$ be a central subgroup of order $p$ such that $K \subseteq G'$. Now $|G/K|=p^{n-1}$ and $|(G/K)'|=p^{k-1}$. Now we get\\
\centerline{ $|M(G/K)| \leq p^{\frac{1}{2}(n-1+k-1-2)(n-1-k+1-1)+1}=p^{\frac{1}{2}(n+k-4)(n-k-1)+1}$.}
Hence the inequality $|M(G)|p \leq |M(G/K)||G^{ab}|$ gives \\
\centerline{$|M(G)|\leq p^{\frac{1}{2}(n+k-4)(n-k-1)+1} p^{(n-k-1)}=p^{\frac{1}{2}(n+k-2)(n-k-1)+1}$.}
It follows that $|M(G/K)|$ attains the bound and $G/G'$ is elementary abelian $p$-group. 

$(ii)$ Suppose the exponent of $Z(G) > p$. Consider a cyclic central subgroup $K$ of order $p^2$. Either $K \subset G'$, $K \cap G'=1$ or $K \cap G'=\mathbb{Z}_p$.

For the first case $|G/K|=p^{n-2}$ and $|(G/K)'|=p^{k-2}$. Hence\\
\centerline{ $|M(G/K)| \leq p^{\frac{1}{2}(n-2+k-2-2)(n-2-k+2-1)+1}=p^{\frac{1}{2}(n+k-6)(n-k-1)+1}$.}
Therefore from $|M(G)|p^2 \leq |M(G/K)||G^{ab}|$ we have\\
\centerline{$|M(G)| \leq p^{\frac{1}{2}(n+k-4)(n-k-1)+1} p^{(n-k-2)}=p^{\frac{1}{2}(n+k-2)(n-k-1)+1-(n-k)}$,}
which is a contradiction.
Other cases follows similarly.

$(iii)$  Suppose $Z(G) \nsubseteq G'$, consider a central subgroup $K$ of order $p$ such that $K \cap G'=1$. Since $|G/K|=p^{n-1}$ and $|(G/K)'|=p^k$. So \\
\centerline{$|M(G/K)| \leq p^{\frac{1}{2}(n+k-3)(n-k-2)+1}$. }
Hence we have\\
\centerline{$|M(G)| \leq |M(G/K)||G/G'K| \leq p^{\frac{1}{2}(n+k-3)(n-k-2)+1} p^{(n-k-1)}$}
\centerline{\hspace{4 cm}$= p^{\frac{1}{2}(n+k-2)(n-k-1)+1-(k-1)}$,}
 which is a contradiction for $k>1$. For $k=1$, $|M(G)|$ attains the bound if and only if $G \cong ES_p(p^3) \times \mathbb{Z}_p^{(n-3)}$. \hfill$\Box$
 
\end{proof}
The following lemma follows from the proof of Lemma \ref{R}$(i)$ using Lemma 3.1$(ii)$.
\begin{lemma}\label{K}
If $G$ is a $p$-group of order $p^n$ such that $|M(G)|$ attains the bound, then for every central subgroup $K$ of order $p$, $|M(G/K)|$ also attains the bound.
\end{lemma} 
\begin{lemma}\label{M}
There is no group $G$ of order $p^n$ $(n \geq 4)$ having maximal class such that $|M(G)|$ attains the bound.
\end{lemma}
\begin{proof}
First we prove that $|M(G)| \leq p^{n-2}$ for $p$-groups $G$ of maximal class. We use induction argument on $n$ to prove this. 

Let $n=4$. Then for $p=2$ using HAP\cite{HAP} of GAP\cite{GAP} and for odd $p$ by \cite[page. 4177]{EG} it follows that $|M(G)| \leq p^2=p^{n-2}$.

Now consider group $G$ of order $p^n$ $(n >4)$ of maximal class. Note that $G/Z(G)$ is also of maximal class. So by induction hypothesis $|M(G/Z(G))| \leq p^{n-3}$. Hence it follows from \cite[Proposition 2.4]{MRR} that \\
\centerline{$|M(G)|p \leq |M(G/Z(G))||G^{ab}| \leq p^{n-1}$.}
Hence $|M(G)| \leq p^{n-2}$. Now if $|M(G)|$ attains the bound for $p$-groups $G$ of maximal class then 
$|M(G)|=p^{\frac{1}{2}(n+n-2-2)(n-n+2-1)+1}=p^{n-1}$, which is a contradiction. \hfill$\Box$

\end{proof}
In view of Lemma \ref{K} we observe that it is sufficient to consider groups $G/K$ such that $|M(G/K)|$ attains the bound for every central subgroup $K$ of order $p$. This observation is going to be key ingredient in the proof.
The following lemma refutes the existence of finite $p$-groups $G$ such that $G/K$ is of nilpotency class $2$ and $|M(G)|$ attains the bound.
\begin{lemma}\label{C2}
There is no non-abelian $p$-group $G$ of order $p^n$, $p\neq 3$ such that $G/K$ is of nilpotency class $2$ for some central subgroup $K$ of order $p$ and $|M(G)|$ attains the bound. For $p=3$, there is a group $G$ such that $|M(G)|$ attains the bound.
\end{lemma}
\begin{proof}
Suppose that $G$ is a group of order $p^n$ and $|G'|=p^k$ such that $|M(G)|$ attains the bound.
Let $G/K$ is of class $2$ for a central subgroup $K$ of order $p$. By Lemma \ref{K} $|M(G/K)|$ also attains the bound. Hence by Theorem \ref{Q} we have $G/K \cong G_1, G_2$ or $G_3$. Now we consider the cases depending on the structure of $G/K$.

If $G/K \cong G_1=ES_p(p^3) \times \mathbb{Z}_p^{(n-3)}$, then $k=2$ and $|M(G)|=p^{\frac{1}{2}(n+k-2)(n-k-1)+1}$ $=p^{\frac{1}{2}n(n-1)-(n-1)}$, which contradicts Theorem \ref{P}.
If $G/K \cong G_2$ then $k=3$ and $|M(G)|=p^8=p^{\frac{1}{2}6(6-1)-(6+1)}$, which contradicts Theorem \ref{S}.

If $G/K \cong G_3=\langle \alpha_1,\beta_1,\alpha_2,\beta_2,\alpha_3, \beta_3 \mid [\alpha_1,\alpha_2]=\beta_3, [\alpha_2,\alpha_3]=\beta_1, [\alpha_3,\alpha_1]=\beta_2,\alpha_i^{(p)}=\beta_i^p=1(i=1,2,3)\rangle$, then $k=4$ and $|M(G)|=$ $p^{\frac{1}{2}(n+k-2)(n-k-1)+1}$ $=p^{10}$. We want to show that $|Image(\psi_3)| \geq p$ for $p\neq 3$. Then from Theorem \ref{RM} we have $|M(G)| \leq p^9$, which is a contradiction.

Suppose $[\beta_i,\alpha_j]$ is non-trivial in $\gamma_3(G)$ for some $j \in \{1,2,3\}$. Without loss of generality assume that $i=1$. Now if  $j=2$ or $3$, then $\psi_3(\alpha_2 \otimes \alpha_3 \otimes \alpha_j \otimes \alpha_1)$ is non-trivial element. Then $|Image(\psi_3)| \geq p$. Similarly we can show for $i=2,3$.

Hence we consider the case $[\beta_i, \alpha_i] \in \gamma_3(G)$ is non-trivial and $[\beta_i,\alpha_j]=1$ for $i,j \in \{1,2,3\}, i \neq j$.
Now let $|Z(G)| > p$. 
Without loss of generality assume that $\beta_2=[\alpha_3,\alpha_1] \in Z(G)$ and there is a non-trivial element $[\beta_1, \alpha_1] \in \gamma_3(G)$. Then  $\psi_3(\alpha_2 \otimes \alpha_3 \otimes \alpha_1 \otimes \alpha_3)$ is non-trivial element. So $|Image(\psi_3)| >p$.
Next case remains $Z(G)=\langle \gamma \rangle \cong \mathbb{Z}_p$. Suppose $Image(\psi_3)=\{1\}$. Then $\psi_3(\alpha_1 \otimes  \alpha_2 \otimes \alpha_3 \otimes \alpha_1)=\psi_3(\alpha_3 \otimes  \alpha_1 \otimes \alpha_2 \otimes \alpha_3)=1$ forces to have $[\beta_3,\alpha_3]=[\beta_2,\alpha_2]=[\beta_1,\alpha_1]$ and by Hall-Witt identity we have $p=3$. 

Using these relations we construct a group $G=\langle \alpha_1,\beta_1,\alpha_2,\beta_2,\alpha_3, \beta_3,\gamma \mid [\alpha_1,\alpha_2]=\beta_3, [\alpha_2,\alpha_3]=\beta_1, [\alpha_3,\alpha_1]=\beta_2, [\beta_3,\alpha_3]=[\beta_2,\alpha_2]=[\beta_1,\alpha_1]=\gamma,
\alpha_i^3=\beta_i^3=\gamma^3=1(i=1,2,3)\rangle$ which is of order $3^7$. Using HAP \cite{HAP} of GAP \cite{GAP} we see that $|M(G)|=p^{\frac{1}{2}(n+k-2)(n-k-1)}+1=p^{10}$.

\hfill$\Box$

\end{proof}
\noindent We are now ready to prove our main result.\\
{\bf Proof of Theorem \ref{C1}.}
First we prove that there is no group $G$ of order $p^n$, $p \neq 3$ of class $c=3$ such that $|M(G)|$ attains the bound. For a central subgroup $K$ of order $p$ in $G$, $G/K$ is of class $2$ or $3$. If $G/K$ is of class $2$, then the result follows from Lemma \ref{C2}.
Now if $G/K$ is of nilpotency class $3$ for every central subgroup $K$ of order $p$, then we use induction on $n$ to prove that there is no such $G$ such that $|M(G)|$ attains the bound.
For $n=5$, $G/K$ is of maximal class, so our result is true for $n=5$, follows from Lemma \ref{M} and Lemma \ref{K}. 
Now let $G$ be a group of class $3$ and $n > 5$. 
If there is a central subgroup of $G/K$  of order $p$ such that the factor group is of class $2$, then the result follows from Lemma \ref{C2} and Lemma \ref{K}.
Hence consider that for every central subgroup of $G/K$  of order $p$, the factor group of $G/K$ is of class $3$ again. Since $|G/K|=p^{n-1}$, by induction hypothesis on $n$ there is no such $G/K$ such that $|M(G/K)|$ attains the bound. Hence by Lemma \ref{K}, result follows for class $c=3$. 

We proved our result for $c=3$. Now we use induction argument to complete the proof for $c > 3$. Let $G$ be a $p$-group of order $p^n$ and of class $c>3$. 
If nilpotency class of $G/K$ is smaller than $c$, then by induction hypothesis on $c$, there is no $G/K$ such that $|M(G/K)|$ attain the bound. Hence result follows by Lemma \ref{K}.
Now if nilpotency class of $G/K$ is $c$ for every central subgroup $K$ of order $p$, then we use induction on $n$ to prove our result.
For $n=c+2$, $G/K$ is of maximal class. So our result is true for $n=c+2$, by Lemma \ref{M} and Lemma \ref{K}.
Let $|G|=p^n$ of class $c$ with $n > c+2$. If there is a central subgroup of $G/K$ of order $p$ such that the factor group is of class smaller than $c$ then the result follows by induction on $c$ and by Lemma \ref{K}.
So now for every central subgroup of $G/K$  of order $p$, the factor group of $G/K$ is of class $c$ again. As $|G/K|=p^{n-1}$, by induction hypothesis on $n$ there is no such $G/K$ such that $|M(G/K)|$ attains the bound. Hence our result follows by Lemma \ref{K}.
This completes the proof.
\hfill$\Box$
\\
\\
We conclude our paper by providing some examples of groups $G$ of order $p^n$, such that $|M(G)|=p^{\frac{1}{2}(n+k-2)(n-k-1)}$ where $|G'|=p^k$.\\
{\bf Example1:} Consider the group $G=\langle \alpha,\alpha_1,\alpha_2,\alpha_3,\alpha_4\mid [\alpha,\alpha_1]=\alpha_2,[\alpha_2,\alpha]=\alpha_3,[\alpha_2,\alpha_1]=\alpha_4,\alpha^p=\alpha_i^p=1 (i=1,2,3,4) \rangle$ from \cite{RJ}. This is group of order $p^5$ with $|G'|=p^3$. Nilpotency class of $G$ is $3$. For $p=5,7,11,13,17$ using HAP \cite{HAP} of GAP \cite{GAP} we obtain $M(G)\cong\mathbb{Z}_p \times \mathbb{Z}_p \times \mathbb{Z}_p$. Note that $|M(G)|=p^{\frac{1}{2}(5+3-2)(5-3-1)}=p^3$.
\\
{\bf Example2:} Consider the group
$G=\langle \alpha,\alpha_1,\alpha_2,\alpha_3,\alpha_4 \mid [\alpha_i,\alpha]=\alpha_{i+1},\alpha^p=\alpha_1^{(p)}=\alpha_{i+1}^{(p)}=1 (i=1,2,3)\rangle$ from \cite{RJ} and $\alpha_{i+1}^{\left(p\right)} = \alpha_{i+1}^p \alpha_{i+2}^{p \choose 2} \cdots\alpha_{i+k}^{p \choose k}... \alpha_{i+p}$. This is group of order $p^5$ with $|G'|=p^3$. Nilpotency class of $G$ is $4$. For $p=5,7,11,13,17$ using HAP \cite{HAP} of GAP \cite{GAP} we obtain $M(G)\cong\mathbb{Z}_p \times \mathbb{Z}_p \times \mathbb{Z}_p$. Note that $|M(G)|=p^{\frac{1}{2}(5+3-2)(5-3-1)}=p^3$.
\\
\\
{\bf Acknowledgement}:
I am very grateful to my supervisor Dr. Manoj K. Yadav for his encouragement and discussions. Also I am thankful to Dr. Pradeep K. Rai for helpful discussions.
I wish to thank the Harish-Chandra Research Institute, the Dept. of Atomic Energy, Govt. of India, for providing excellent research facility.

\end{document}